\documentclass{amsart}
\usepackage[english]{babel}
\usepackage{amsthm}
\usepackage{inputenc}
\usepackage{amssymb}
\usepackage{graphicx}
\newtheorem{thm}{Theorem}[section]

\newtheorem{lem}[thm]{Lemma}
\newtheorem{prop}[thm]{Proposition}
\newtheorem{defn}[thm]{Definition}
\newtheorem{exam}[thm]{Example}

\numberwithin{equation}{section}

\def\ll{\mathcal{Z}}
\begin{document}

\title[$p$-filiform Zinbiel algebras]{$p$-filiform Zinbiel algebras}%

\author{L.M. Camacho, E.M. Ca\~{n}ete, S. G\'{o}mez-Vidal, B.A. Omirov}
\address{[L.M. Camacho -- E.M. Ca\~{n}ete -- S. G\'{o}mez-Vidal] Dpto. Matem\'{a}tica Aplicada I.
Universidad de Sevilla. Avda. Reina Mercedes, s/n. 41012 Sevilla.
(Spain)} \email{lcamacho@us.es --- elisacamol@us.es  --- samuel.gomezvidal@gmail.com}
\address{[B.A. Omirov] Institute of Mathematics and Information Technologies
 of Academy of Uzbekistan, 29, F.Hodjaev srt., 100125, Tashkent (Uzbekistan)}
\email{omirovb@mail.ru}

%

\begin{abstract}
The paper deals with the classification of a subclass of finite-dimensional Zinbiel algebras: the naturally graded $p$-filiform Zinbiel algebras.
A Zinbiel algebra is the dual to Leibniz algebra in Koszul sense.
We prove that there exists, up to isomorphism, only one family of naturally graded p-filiform Zinbiel algebras under hypothesis $n-p \geq 4.$
\end{abstract}
\maketitle
\section{Introduction.}
Non associative algebras appear at the beginning of the twentieth century due to the development of quantum mechanics. As McCrimmon mentioned in \cite{McC}, physicists needed new mathematics objects different from the associative algebras, that is, objects different from complex matrix. In this direction we would highlight many important algebras as Lie, Jordan, alternative, Leibniz algebras and others.

A Leibniz algebra, as introduced by Loday \cite{loday1}, is a vector space together with a bilinear binary operation $[-,-]$ for which $[-,z]$ is a derivation for every $z$ in the vector space. Thus, Leibniz algebras are noncommutative version of Lie algebras, which are Leibniz algebras whose brackets are skew-symmetric.

Koszul algebras were originally defined by Priddy in 1970 \cite{priddy} and have since revealed important applications in algebraic geometry, Lie theory, quantum groups, algebraic topology and, recently, combinatorics \cite{hohai}. The rich structure and long history of Koszul algebras are clearly detailed in \cite{poli}. There exist numerous equivalent definitions of a Koszul algebra (see for instance \cite{backelin}).
Such an algebra may be understood as a positively graded algebra that is "as close to semisimple as it can possible be" (see \cite{beilinson} for more details). Many nice
homological properties of Koszul algebras have been shown in research areas of
commutative and noncommutative algebras, such as algebraic topology, algebraic
geometry, quantum group and Lie algebra (\cite{aquino}, \cite{beilinson}, etc.).

The Leibniz algebras form a Koszul operad in the sense of V. Ginzburg and M. Kapranov \cite{Gin}. Under the Koszul duality the operad of Lie algebras is dual to the operad of associative and commutative algebras. The notion of dual Leibniz algebra defined by J.-L. Loday \cite{Loday} is precisely the dual operad of Leibniz algebras in this sense.

In this paper we will study algebras which are the dual to Leibniz algebras in Koszul sense. J.-L. Loday studied in \cite{Loday2} categorical properties of Leibniz algebras and considered in this connection a new object -- Zinbiel algebras (Leibniz is written in reverse order). Since the category of Zinbiel algebras is Koszul dual to the  category of Leibniz algebras, sometimes they are also called dual Leibniz algebras.
Any dual Leibniz algebra with respect to the symmetrization $a\star b=ab+ba$ is an associative and commutative algebra. For more details we refer to reader to \cite{Loday, libroLoday}. This define a functor $(dual\ Leibniz)\longrightarrow (Com)$ between the categories of algebras, which is dual to the inclusion functor $(Lie)\longrightarrow (Leibniz)$.

The notion of associative dialgebra defines an algebraic operad $Dias,$ which is binary and quadratic. By the theory of Ginzburg and Kapranov \cite{Gin}, there is a well-defined ``dual operad" $Dias^{!}.$ In \cite{libroLoday} Loday and others show that this is precisely the operad $Dend$ of the ``dendriform algebras". Moreover in this paper \cite{Gin} shows that $As^{!}=As,$ $Com^{!}=Lie$ and $Lie^{!}=Com.$

These interrelations are showed in the following categorical diagram, where the main property is that the categories which lie symmetrically with respect to vertical linea crossing the category of associative algebras are Koszul dual categories.\\
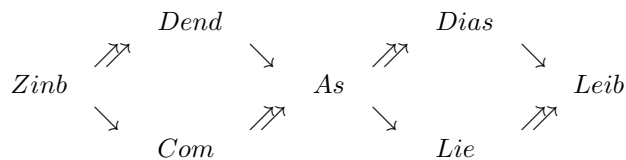
$$\begin{array}{lllllllll}
& & Dend & & & &Dias& &  \\
& \displaystyle\nearrow\hspace{-0.2cm}\nearrow & & \displaystyle\searrow & &\displaystyle\nearrow\hspace{-0.2cm}\nearrow& &\displaystyle\searrow &\\
Zinb& &  & &As  & & &\qquad Leib\\
& \displaystyle\searrow& & \displaystyle\nearrow\hspace{-0.2cm}\nearrow & &\displaystyle\searrow& &\displaystyle\nearrow\hspace{-0.2cm}\nearrow &\\
& & Com & & & &Lie& &
\end{array}$$

The symbol $A \rightrightarrows B$ means that the algebra of
category A is in the category B and the symbol $A\longrightarrow
B$ indicates that the algebra of category A with a special
operation gives an algebra of the category B.

In the works \cite{Jobir3, Dz, Umir} some interesting properties of Zinbiel algebras were obtained. In particular, the nilpotency of an arbitrary complex finite-dimensional Zinbiel algebra was proved. Thus, the classification of finite-dimensional complex Zinbiel algebras is reduced to nilpotent ones. However, the study of the non-associative nilpotent algebras, in particular of the Zinbiel algebras, is too complex. In fact, the same study in Lie case was appeared two centuries ago and it is still unsolved. Therefore  we reduce our attention to an important family: the naturally graded $p$-filiform Zinbiel algebras, since it gives a really crucial information about p-filiform algebras without gradation restriction.

Let us introduce some definitions and notations, all of them necessary for the understanding of this work.

\begin{defn}
A Leibniz algebra over $K$ is a vector space $\mathcal{L}$ equipped with a bilinear map, called bracket,
$$[-,-]:  \mathcal{L} \times  \mathcal{L} \rightarrow  \mathcal{L} $$
satisfying the Leibniz identity:
$$[x,[y,z]]=[[x,y],z]-[[x,z],y],$$
for all $x,y,z \in  \mathcal{L}.$
\end{defn}
\begin{exam}
Any Lie algebra is a Leibniz algebra.
\end{exam}

\begin{defn} A vector space $\ll$ with a bilinear operation ``$\circ$'' is called Zinbiel algebra if for any $x,y,z \in \ll$ the following identity
$$ (x\circ y)\circ z=x\circ (y\circ z)+x\circ (z \circ y)$$
holds.
\end{defn}

Examples of Zinbiel algebras can be found in \cite{Jobir3,Dz,Loday}.

For a given Zinbiel algebra $\ll$ we define the following sequence
$$\ll^1=\ll, \ \ll^{k+1}=\ll \circ \ll^k, \ k\geq 1. $$

\begin{defn}
A Zinbiel algebra $\ll$ is called nilpotent if there exists $s \in \mathbb{N}$ such that $\ll^s \neq 0$ and $\ll^{s+1}=0.$ The minimal number $s$ satisfying this property is called the index of nilpotency or nilindex of the algebra $\ll.$
\end{defn}

It is not difficult to see that the nilindex of an arbitrary $n$-dimensional nilpotent algebra does not exceed the number $n$.

The set $R(\ll) = \{x \in \ll \ | \  x \circ y = 0 \hbox{ for any } y \in \ll\}$ is called the right annihilator of
the Zinbiel algebra $\ll,$ the set $L(\ll) = \{x \in \ll \ | \ y \circ x = 0 \hbox{ for any } y \in \ll\}$ is the left annihilator and $Cent(\ll) = \{x,y \in \ll \ | \ x \circ y= y \circ x = 0 \hbox{ for any } y \in \ll\}$ is the center of $\ll.$

Let us denote by $L_{x}$ the left operator $L_x: \ll \longrightarrow \ll$ such that $L_x(y)=[x,y]$ for any $y \in \ll.$

Let $x$ be a nilpotent element of the set $\ll \setminus \ll^2$. For the nilpotent operator $L_x$ we define a descending sequence $C(x)=(n_1,n_2, \dots, n_k)$, which consists of the dimensions of the Jordan blocks of the operator $L_x$. In the set of such sequences we consider the lexicographic order, that is,
  $C(x)=(n_1,n_2, \dots, n_k)< C(y)=(m_1, m_2, \dots, m_s)$ if and only if there exists $i \in \mathbb{N}$ such that $n_j=m_j$ for any $j<i$ and $n_i<m_i$.

  \begin{defn}\label{def:char.seq}
  The sequence $C(\ll)=\max \{C(x)\  : \ x \in \ll \setminus \ll^2 \}$ is called the characteristic sequence of the algebra $\ll$.
  \end{defn}

  Let $\ll$ be an $n$-dimensional nilpotent Zinbiel algebra and $p$ a non negative integer ($p<n$).
  \begin{defn}\label{def:p-fili}
  The Zinbiel algebra $\ll$ is called $p$-filiform if $C(\ll)=(n-p,\underbrace{1,\dots,1}_{p})$. If $p=0$, $\ll$ is called null-filiform
  and if $p=1$ it is called filiform.
  \end{defn}

  Let $\ll$ be a finite-dimensional nilpotent Zinbiel algebra. Put
$\ll_i=\ll^i/\ll^{i+1}$ with $1 \leq i \leq s-1,$ where $s$ is the nilindex of
the algebra $\ll$ and denote $gr(\ll) = \ll_1 \oplus \ll_2\oplus
\ldots\oplus \ll_{s-1}.$ The graded Zinbiel algebra,
$gr(\ll),$ is obtained where $[\ll_i,\ll_j] \subseteq \ll_{i+j}.$ An algebra $\ll$ is called naturally graded if $\ll\cong gr(\ll).$

Let $\ll$ be a naturally graded Zinbiel algebra with characteristic
sequence $(n-p,1, \dots,1).$ By de\-finition of characteristic sequence
there exists a basis $\{e_1, e_2, \dots, e_n\}$ in the algebra $\ll$
such that the operator $L_{e_1}$ has one block $J_{n-p}$ of
size $n-p$ and the others blocks $J_1$ of
size one.

The aim of this work is to continue the study of naturally graded Zinbiel algebras. Since the null-filiform, filiform and 2-filiform cases are solved in \cite{TesisJobir, Jobir3}, we give a step further and we obtain the classification in the general case of $p$-filiform algebras. Moreover, the thesis of Adashev \cite{TesisJobir} allows us to reduce our discussion to the case $n-p \geq 4.$

The paper is divided into two sections. In the first one we establish the natural gradation of the naturally graded $p$-filiform Zinbiel algebras. In addition, we obtain main information on the structural constants of the law of our algebras. In the second section we present the classification of these kind of algebras (Theorem \ref{classif}), where we have proved that there is only one family of $n$-dimensional naturally graded $p$-filiform Zinbiel algebras with $n-p \geq 4.$

\section{Natural gradation of $p$-filiform Zinbiel algebras}
The classification of the naturally graded $p$-filiform Zinbiel algebras presented in this work is esta\-blished using the characteristic sequence as the principal invariant.
Let $Z$ be a graded $p$-filiform $n$-dimensional Zinbiel algebra.
Then there exists a basis $\{e_1, e_2, \dots, e_{n-p}, f_1, \dots,
f_p\}$ such that $e_1\in Z\setminus Z^2$ and $C(e_1)=(n-p,
\underbrace{1,\dots,1}_{p-times}).$ By the definition of
characteristic sequence the operator $L_{e_1}$ in the Jordan form
has one block $J_{n-p}$ of size $n-p$ and $p$ blocks $J_1$ (where
$J_1=\{0\}$) of size one, that is, the left operator $L_{e_1}$ is isomorphic to one of the following matrix:

$$\left(\begin{array}{cccc}
J_{n-p}&0&\dots&0\\
0&J_1&\dots&0\\
\vdots&\vdots&\ddots&\vdots\\
0&0&\dots&J_1\\
\end{array}\right),
\left(\begin{array}{cccc}
J_1&0&\dots&0\\
0&J_{n-p}&\dots&0\\
\vdots&\vdots&\ddots&\vdots\\
0&0&\dots&J_1\\
\end{array}\right),\dots,
\left(\begin{array}{cccc}
J_{1}&\dots&0&0\\
\vdots&\ddots&\vdots&\vdots\\
0&\dots&J_1&0\\
0&\dots &0& J_{n-p}\\
\end{array}\right).$$

Making basis shift one can assume that there are only two following types:
$$\left(\begin{array}{cccc}
J_{n-p}&0&\dots&0\\
0&J_1&\dots&0\\
\vdots&\vdots&\ddots&\vdots\\
0&0&\dots&J_1\\
\end{array}\right),
\left(\begin{array}{cccc}
J_1&0&\dots&0\\
0&J_{n-p}&\dots&0\\
\vdots&\vdots&\ddots&\vdots\\
0&0&\dots&J_1\\
\end{array}\right).$$

A $p$-filiform Zinbiel algebra is called of first type if the operator $L_{e_1}$ is isomorphic to the first matrix. Otherwise it is called of second type.

It is worthwhile to consider the following result since it allow us to reduce our study to the algebras of first type.

\begin{prop} \label{prode} Let $\ll$ be a $p$-filiform $n$-dimensional Zinbiel algebra of second type. Then $n-p \leq 3.$
\end{prop}
\begin{proof}
The proof is straightforward by definition of second type algebras. Let us see it. By definition of second type algebras there exists a basis $\{e_1, e_{2},\dots,e_{n-p},f_1, \dots,f_p \}$ of $\ll$ such that
$$\begin{array}{ll}
e_1\circ e_1=0,&\\
e_1\circ e_i=e_{i+1},& 2\leq i \leq n-p-1,\\
e_1 \circ e_{n-p}=f_1,&\\
e_1\circ f_j=0, & 1\leq j \leq p,
\end{array}$$
and the other products equal zero. Therefore, the natural gradation of $\ll$ is $$\langle e_1,e_2,f_{2}, \dots,f_{p}\rangle \subseteq Z_1,\ \langle e_3 \rangle \subseteq Z_2,\dots,\langle e_{n-p} \rangle \subseteq Z_{n-p-1},\ \langle f_{1}\rangle \subseteq Z_{n-p}.$$

Assume that $n-p\geq 4$ and let us write $e_2 \circ e_1 =\alpha e_3 +\displaystyle\sum_{i=4}^{n-p} \gamma_i e_i+\displaystyle\sum_{i=1}^p \beta_i f_i.$

Then,
$$\begin{array}{ll}
0=(e_1 \circ e_1) \circ e_2&=e_1\circ (e_1 \circ e_2)+e_1 \circ (e_2 \circ e_1)=e_1\circ e_3 +\alpha e_1 \circ e_3+
\displaystyle\sum_{i=4}^{n-p} \gamma_i e_1 \circ e_i=\\[2mm]
&=(1+\alpha)e_4+\displaystyle\sum_{i=4}^{n-p} \gamma_i e_{i+1},
\end{array}$$ which
implies $\alpha=-1,$ $ \gamma_k=0,$  with $ 4\leq k \leq n-p.$ Hence we have $$e_3 \circ e_1=(e_1 \circ e_2)\circ e_1=e_1 \circ (e_2 \circ e_1)+e_1\circ (e_1 \circ e_2)=(e_1\circ e_1) \circ e_2=0.$$
Therefore, $0=(e_1 \circ e_1)\circ e_3 =e_1\circ (e_1 \circ e_3)+e_1\circ (e_3 \circ e_1)=e_5,$ which is a contradiction.

\end{proof}

The aim of this section is obtain the expression of the natural gradation of the naturally graded $p$-filiform Zinbiel algebras of first type. These information will be useful to come to the classification of our algebras.

From now on we denote by $C_i^j$ the combinatorial numbers
$C_i^j=\left( \begin{array}{c}
   i\\
   j
   \end{array}\right)$

The following result holds for each naturally graded $p$-filiform Zinbiel algebras of first type.

\begin{lem}\label{lemma1} Let $\ll$ be a Zinbiel algebra such that $e_1 \circ e_i=e_{i+1}\hbox{ for } \ 1\leq i \leq k-1,$ with respect to the adapted basis $\{ e_1,\dots,e_k,e_{k+1},\dots,e_n\}.$ Then $$ e_i \circ e_j=C^{j}_{i+j-1}e_{i+j} \quad  \hbox{ for } \ 2\leq i+j \leq k. \eqno(1)$$
\end{lem}

\begin{proof} In order to get $(1)$, it is convenient to prove the equality $e_{i}\circ e_1=ie_{i+1}$ for $1\leq i \leq k-1,$ which is clear by induction.

Assume that $(1)$ holds for $j$ such that  $j\leq k-1$ and $1\leq i\leq k-j$ and let us prove it for $j+1.$ Since
$e_i \circ e_{j+1}=e_i\circ (e_1\circ e_j)=(e_i\circ e_1)\circ e_j-e_i\circ (e_j\circ e_1)= ie_{i+1}\circ e_j-je_i\circ e_{j+1}=iC^j_{i+j}e_{i+j+1}-je_i\circ e_{j+1},$ where $C^j_i$ is the corresponding combinatory number, we conclude that
$$e_i\circ e_{j+1}=\frac{i}{j+1}\cdot C^j_{i+j}e_{i+j+1}=\frac{(i+j)!}{(j+1)!(i-1)!}e_{i+j+1}=C^{j+1}_{i+j}e_{i+j+1}.$$
Therefore, $(1)$ holds for all $2\leq i+j\leq k.$
\end{proof}

Let $\ll$ be a naturally graded $p$-filiform Zinbiel algebra of first type of dimension equals $n$ and let $\{e_1, \dots,e_{n-p},f_1, \dots.f_p\}$ be an adapted bases such that $e_1$ is a characteristic vector of $\ll$. Due to the above lemma and the definition of an algebra of first type we know the following products:
$$\begin{cases}
e_1 \circ e_i=e_{i+1}, \quad 1 \leq i \leq n-p-1,\\
e_1 \circ f_i=0, \quad 1 \leq i \leq p,\\
e_i \circ e_j=C_{i+j-1}^{j}e_{i+j}, \quad 2 \leq i+j \leq n-p,
\end{cases}$$
and the following information on the gradation: $\langle e_1\rangle\subseteq \ll_1, \langle e_2\rangle\subseteq \ll_2, \dots, \langle e_{n-p}\rangle\subseteq \ll_{n-p}.$

  Denote by $r_i$ the number such that $f_i\in \ll_{r_i},$ for $1\leq i\leq p.$  Then one can assume that $1\leq r_1\leq r_2\leq\dots\leq r_p\leq n-p.$ Otherwise we can make a change of basis and to interchange the roles of $f_i.$

\begin{thm} Let $\ll$ be a $n$-dimensional naturally graded  $p$-filiform Zinbiel algebra of first type. Then $r_s\leq s$ for all $1\leq s \leq p.$
\end{thm}

\begin{proof}
First of all, note that $r_1=1.$ Indeed if $r_1>1$ then $\ll$ would be one-generated, that is $\ll$ would be null-filiform Zinbiel algebra. That implies that its characteristic sequence is $C(\ll)=(n),$ which contradicts our hypothesis.

We now proceed by induction on $s$. Let us see the restriction on $r_2$. Let us suppose $r_2>2,$ then $\ll_1=\langle e_1,f_1\rangle,\, \ll_2=\langle e_2 \rangle,\dots, \ll_{r_2}=\langle e_{r_2},f_2\rangle$ which implies $\ll_{r_2}=\ll_1\circ \ll_{r_2-1}=\langle e_1, f_1\rangle \circ \langle e_{r_2-1}\rangle =\langle e_{r_2}, f_1\circ e_{r_2-1}\rangle.$

Consider $f_1 \circ e_{r_2-1}=f_1\circ (e_1\circ e_{r_2-2})=(f_1\circ e_1)\circ e_{r_2-2}-f_1\circ (e_{r_2-2}\circ e_1).$ Since $f_1\circ e_1 \in \ll_2$ we put $f_1 \circ e_1=\alpha e_2$ for some $\alpha \in \mathbb{C}.$

Then
$$\begin{array}{ll}
f_1 \circ e_{r_2-1}&=\alpha e_2 \circ e_{r_2-2}-(r_2-2)f_1\circ e_{r_2-1}=\alpha C_{r_2-1}^{r_2-2}e_{r_2}-(r_2-2)f_1\circ e_{r_2-1}=\\
&=\alpha(r_2-1)e_{r_2}-(r_2-2)f_1\circ e_{r_2-1}
\end{array}$$
which implies $(r_2-1)f_1\circ e_{r_2-1}=\alpha(r_2-1)e_{r_2},$ such that $\ll_{r_2}=\langle e_{r_2} \rangle$ which contradicts $f_2 \in \ll_{r_2}$. Thus, $r_2\leq 2.$

Assume that $r_k\leq k$ for all $1\leq k \leq s-1$ for some $s\in\mathbb{N};$ we will prove that $r_s \leq s.$ Suppose, contrary to our claim, that $r_s>s.$ To come to a contradiction, we first show that
$$f_t \circ e_{r_s-r_t}\subseteq \langle e_{r_s}\rangle, \ 1\leq t \leq s-1. \eqno (2)$$

If $t=s-1$ then $f_{s-1}\circ e_{r_s-r_{s-1}}=f_{s-1}\circ (e_1\circ e_{r_s-r_{s-1}-1})=(f_{s-1}\circ e_1)\circ e_{r_s-r_{s-1}-1}-f_{s-1}\circ (e_{r_s-r_{s-1}-1}\circ e_1).$ Moreover since $f_{s-1}\circ e_1\in \ll_{r_{s-1}}\circ \ll_1\subseteq \ll_{r_{s-1}+1}=\langle e_{r_{s-1}+1}\rangle$ we can write $f_{s-1}\circ e_1=\beta e_{r_{s-1}+1}$ for some $\beta \in \mathbb{C}.$ Therefore,
$$f_{s-1}\circ e_{r_s-r_{s-1}}=\beta C_{r_s-1}^{r_s-r_{s-1}-1}e_{r_s}-(r_s-r_{s-1}-1)f_{s-1}\circ e_{r_s-r_{s-1}},$$ that is
$$(r_s-r_{s-1})f_{s-1}\circ e_{r_s-r_{s-1}}=\beta e_{r_{s-1}+1},$$ which satisfies the inclusion $(2)$ when $t=s-1.$ We now see that $f_t \circ e_{r_s-r_t} \subseteq \langle e_{r_s} \rangle  \hbox{ with }1 \leq t \leq s-2.$ Since $f_t\circ e_{r_s-r_t}=f_{t}\circ (e_1\circ e_{r_s-r_{t}-1})=(f_{t}\circ e_1)\circ e_{r_s-r_{t}-1}-f_{t}\circ (e_{r_s-r_{t}-1}\circ e_1)$ and $f_{t}\circ e_1\in \ll_{r_{t}}\circ \ll_1\subseteq \ll_{r_{t}+1},$ it is necessary distinguish the following cases:

\

\begin{itemize}
\item {Case 1: If $r_t+1=r_{t+1}$} we have  $\ll_{r_{t}+1}=\langle e_{r_{t+1}},f_{t+1}\rangle.$ From this we conclude that $f_{t+1} \circ  e_{r_s-r_t-1} \in \langle e_{r_s} \rangle$ and $e_{r_t+1} \circ  e_{r_s-r_t-1} \in \langle e_{r_s} \rangle$, hence that $(r_s-r_t)f_t \circ e_{r_s-r_t} \in \langle e_{r_s} \rangle.$ And, in consequence, $f_{t}\circ (e_1\circ e_{r_s-r_{t}-1})=f_{t+1} \circ e_{r_s-r_{t+1}} \in \langle e_{r_s} \rangle.$ Therefore we assert that $f_t \circ e_{r_s-r_t} \in \langle e_{r_s} \rangle.$

    \

\item{Case 2: If $r_t+1<r_{t+1}$} then $f_{t}\circ e_1 \in \langle e_{r_t+1} \rangle=\ll_{r_t+1}$  and by using similar arguments than above case we obtain $f_t \circ e_{r_s-r_t} \in \langle e_{r_s} \rangle$.
\end{itemize}

\

Therefore $(2)$ is proved.

Now $\ll_1\circ \ll_{r_s-1}=\ll_{r_s}=\langle e_1,f_1,\dots,f_q\rangle \circ \langle e_{r_{s-1}}\rangle$ and since $r_1=\dots=r_q=1$ by $(2)$ we obtain $f_i \circ e_{r_s-1}\in \langle e_{r_s}\rangle$ for $1\leq i \leq q.$ This implies that $\ll_{r_s}=\langle e_{r_s}\rangle,$ a contradiction. Consequently, $r_s\leq s.$

\end{proof}
An important property of our gradation is showed in the following lemma: it is not possible that all $f_i$ with $1 \leq i \leq p$ are generators, otherwise the algebra would be split.

\begin{lem} Let $n-p \geq 3$ and $r_i=1$ for $1\leq i \leq p.$ Then $\ll$ is split.
\end{lem}

\begin{proof} Under these hypothesis we have
$$\ll_1=\langle e_1,f_1,\dots,f_p\rangle, \ \ll_2=\langle e_2 \rangle, \ \dots ,\  \ll_{n-p}=\langle e_{n-p} \rangle,$$  where
$e_1\circ e_i=e_{i+1}, \hbox{ for } 1\leq i \leq n-p-1 \hbox{ and }\ e_1\circ f_j=0 \hbox{ for } 1 \leq j \leq p.$

Let us prove that $f_i \in Cent(\ll)$ for $1 \leq i \leq p.$ By properties of the gradation we can write $f_i\circ e_1=\alpha_i e_2$ and $f_i \circ f_j=\beta_{ij}e_2$ for $1\leq i,j\leq p.$ Since $n-p \geq 3$ we have $$0=(e_1\circ f_i)\circ e_1=e_1\circ(f_i\circ e_1)+ e_1 \circ (e_1\circ f_i)=e_1\circ(f_i\circ e_1)=\alpha_ie_3,$$
hence $\alpha_i=0$ for $1 \leq i \leq p.$

Now $f_i \circ e_{j+1}=f_i \circ (e_1 \circ e_j)=(f_i\circ e_1)\circ e_j-f_i\circ (e_j\circ e_1)=-jf_i\circ e_{j+1}$ which implies $f_i\circ e_{j+1}=0$ for $1\leq j \leq n-p-1.$

From $(f_i\circ f_j)\circ e_1=f_i\circ (f_j\circ e_1)+f_i\circ (e_1\circ f_j)=0$ it follows that $2\beta_{ij}e_3=0,$ that is  $\beta_{ij}=0 \hbox{ for all }  1\leq i,j \leq p.$

Finally it is clear that $e_j \circ f_i=0$ by induction on $j,$ because $e_1 \circ f_i=0 \hbox{ for all } i$ by nilpotence. That implies that $\ll=NF_{n-p}\oplus \mathbb{C}^p,$ where $\mathbb{C}^p=\mathbb{C}\oplus\dots\oplus \mathbb{C}=\langle f_1 \rangle \oplus \dots \oplus \langle f_p \rangle$ and $NF_{n-p}$ is a naturally graded $p$-filiform Zinbiel algebra of dimension $n-p.$
 \end{proof}

An easy computation,  make it obvious the following result.

\begin{lem}\label{lemma}Let $\ll$ be a $p$-filiform Zinbiel algebra of first type. Assume that $f_j\circ e_i=0$ for some $1\leq i \leq n-p-2 \hbox{ and } 1\leq j \leq p.$ Then $f_j\circ e_{i+1}=\dots=f_j\circ e_{n-p-1}=0.$
\end{lem}


It is worthwhile to use the notation $\ll_1=\langle e_1,f_1, \dots,f_{s_1}\rangle$ and $\ll_i=\langle e_i,f_{s_1+\dots+s_{i-1}+1},\dots, f_{s_1+\dots+s_i}\rangle$ for $2\leq i\leq n-p,$ where $s_1,\dots,s_{n-p}$ are non-negative integers such that $s_1+\dots+s_{n-p}=p$ and $dim(\ll_i)=s_i+1$ for $1 \leq i \leq n-p.$ The existence of $r_i>1$ in this terms is equivalent to $s_1<p.$

The distribution of the vectors $f_i$ in the natural gradation is given by the following result.

\begin{prop}\label{s<s} Let $\ll$ be an $n$-dimensional $p$-filiform Zinbiel algebra of first type. Then $0\leq s_{n-p}\leq\dots\leq s_2\leq s_1<p.$
\end{prop}

\begin{proof} We have divided the proof into four steps.
\begin{itemize}
\item{Step $I$:} First of all we will deal with the proof of the following affirmations, where the two first will be proved by induction on $q$
$$ \left\{
\begin{array}{ll}
e_q\circ f_j=0,  &  1\leq q \leq n-p, \ 1\leq j \leq s_1\\
f_j\circ e_q\in \langle f_{s_1+1},\dots,f_p\rangle, &  1\leq q \leq n-p-2, \ 1\leq j \leq s_1\\
f_j\circ e_{n-p-1} \in \langle e_{n-p},f_{s_1+1},\dots,f_p\rangle, & 1\leq j \leq s_1\\
\end{array}\right. \eqno(3)$$

For $q=1$ it is clear that $e_1 \circ f_j=0$ for $1 \leq j \leq p$ by properties of nilpotence. Moreover by means of the natural gradation it is true that $f_j \circ e_1 \in \ll_2$ for $1 \leq j \leq s_1,$ then
 $$\displaystyle f_j\circ e_1=\alpha_j e_2 +\sum_{i=1}^{s_2}\beta_{j,i}f_{s_1+i}, \quad 1 \leq j \leq s_1.$$

 Let us see $\alpha_j=0 \hbox{ for any } j.$ Since $\displaystyle 0=(e_1\circ f_j)\circ e_1=e_1\circ (f_j\circ e_1)+e_1\circ (e_1\circ f_j)=e_1\circ (\alpha_j\circ e_2) +\sum_{i=1}^{s_2}\beta_{i,j}f_{s_1+i})=\alpha_j e_3,$ hence $\alpha_j=0$ for $1\leq j \leq s_1.$
Therefore, we obtain $\displaystyle f_j \circ e_1=\sum_{i=1}^{s_2}\beta_{j,i}f_{s_1+i}\in \langle f_{s_1+1},\dots,f_{s_1+s_2}\rangle$ for $1\leq j \leq s_1.$

Assuming $(3)$ to hold for $q$, we will prove it for $q+1.$ Then we obtain
 $$e_{q+1}\circ f_j=(e_1\circ e_q)\circ f_j=e_1\circ (e_q\circ f_j)+e_1\circ (f_j\circ e_q)=\sum_{i=1}^{s_2}\beta_{j,i}e_1 \circ f_{s_1+i}=0.$$

Moreover the property $f_j \circ e_{q+1} \in \langle f_{s_1+1}, \dots,f_p \rangle$ is obtained by considering the following product:
$0=(e_1\circ f_j)\circ e_{q+1}=e_1\circ (f_j\circ e_{q+1})+e_1\circ (e_{q+1}\circ f_j)=e_1\circ (f_j\circ e_{q+1}).$ Therefore it is clear that $f_j\circ e_{q+1}\in L_{e_1}$, that is $f_j\circ e_{q+1}\in \langle e_{n-p}, f_{s_1+1},\dots,f_p\rangle$ for $1\leq j \leq s_1.$ Note that, if $q+1\neq n-p-1,$ then $f_j\circ e_{q+1}\in \langle f_{s_1+1},\dots,f_p\rangle.$

 By similar way we prove that $e_{n-p-1}\circ f_j=0$ and $f_j\circ e_{n-p-1}\in \langle e_{n-p}, f_{s_1+1},\dots,f_p\rangle$ for $1\leq j \leq s_1.$  Obviously $e_{n-p}\circ f_j=0$ because $e_{n-p} \in Cent(\ll).$ Hence $(3)$ is proved.

 \

 \

 \item{Step $II$:} The goal in this step is to check the antisymmetric property of the products $f_i \circ f_j.$ Let us denote $\displaystyle f_i\circ f_j=\gamma_{ij} e_2+\sum_{t=1}^{p-s_1} \eta_{ij}^t f_{t+s_1},\ \hbox{ for } 1\leq i,j\leq s_1,$ then we can write
 $$\begin{array}{ll}
 0&=(e_1 \circ f_i)\circ f_j=e_1\circ (f_i\circ f_j)+e_1\circ (f_j\circ f_i)=\\
 &=\gamma_{ij}e_1 \circ e_2+ \displaystyle\sum_{t=1}^{p-s_1}\eta_{ij}^t e_1 \circ f_{t+s_1}+ \gamma_{ji}e_1 \circ e_2 + \displaystyle \sum_{t=1}^{p-s_1}\eta_{ji}^t e_1 \circ f_{t+s_1}=\\
 &=(\gamma_{ij}+ \gamma_{ji})e_3,
 \end{array}$$
and so $\gamma_{ij}=-\gamma_{ji}$ holds for $1 \leq i,j \leq s_1.$

\

\

\item{Step $III$:} Our next objective is to prove that $f_i \circ f_j \in \langle e_2 \rangle$  $\hbox{for all } i,j,$ that is, $s_2 \leq s_1.$ On the contrary, suppose that there exists $i,j,t\in \mathbb{N}$ such that $\eta_{ij}^t\neq 0,$ in the case $1 \leq i,j \leq s_1.$ Then making the basis transformation $e_1'=Ae_1+Bf_i$ with $A \neq 0$ and $B\neq 0$ we obtain the
  following linearly independent vectors:
  \begin{align*}
  e_2'&=e_1'\circ e_1'=A^2e_2+g_1(A,B,f_{s_1+1},\dots,f_p),\\
  e_i'&=e_1'\circ e_{i-1}'=A^ie_i+g_{i-1}(A,B,f_{s_1+1},\dots,f_p),
  \end{align*}
  where $g_{i-1}$ is a polynomial with $deg_A(g_{i-1})\leq i-1$ for $3 \leq i \leq n-p.$  Note that, $e_1'\circ e_{n-p}'=0$ and
  $\displaystyle e_1'\circ f_j=(Ae_1+Bf_i)\circ f_j=Ae_1 \circ f_j+Bf_i\circ f_j=B\gamma_{ij}e_2+B\sum_{t=1}^{p-s_1}\eta_{ij}^tf_{s_1+t}.$
  If $\eta_{ij}^t\neq 0,$ then the rank of $L_{e_1'}$ would be greater than $n-p-1,$ which contradicts the assumption $C(\ll)=(n-p,1\dots,1).$ Hence,
$$f_i\circ f_j=-f_j\circ f_i=\gamma_{ij}e_2 \ \hbox{ for }\ 1 \leq i,j \leq s_1.\eqno(4)$$

Analogously one can establish
$$f_i \circ f_j =-f_j\circ f_i \in \langle e_{k+1} \rangle \hbox{ for }  1\leq i\leq s_1 \hbox{ and } s_1+\dots+s_{k-1}+1\leq j \leq s_1+\dots+s_k.\eqno(5)$$

Now from $(3),(4)$ and $(5)$ we have $\ll_2=\ll_1\circ \ll_1=\langle e_2,f_1\circ e_1,\dots,{f_{s_1}}\circ e_1\rangle \subseteq \langle e_2,f_{s_1+1}, \dots, f_{p}\rangle$
therefore, $\dim \ll_2\leq \dim \ll_1=s_1+1$ and $s_2\leq s_1.$

\

\

\item{Step $IV$:} It is worthwhile to check $s_1+\dots+s_k < p$ yields $s_{k+1} \neq 0$ to prove the proposition. On the one hand let us see for $k=1.$ Indeed,
if $s_2=0$ it is clear that  $f_i\circ e_1=0$ for $1 \leq i \leq s_1.$ Moreover by Lemma \ref{lemma} we assert $f_j\circ e_i=0$ for all
$1\leq i\leq n-p \ \hbox{ and }\ 1\leq j\leq s_1.$ Hence $\ll_3=\ll_1\circ \ll_2=\langle e_1,f_1,\dots,{f_{s_1}}\rangle \circ \langle e_2\rangle =\langle e_3\rangle,$
which implies $\ll_i= \langle e_i\rangle$ for $3 \leq i \leq n-p.$ From these assertions one concludes that $s_1=p,$ which is not possible.

On the other hand, let us prove that $s_1+\dots+s_k<p$ yields $s_{k+1}\neq 0,$ since otherwise $\ll_{k+1}=\langle e_{k+1}\rangle.$ As $\ll_{k+1}=\ll_1\circ \ll_{k}$ we obtain
  $f_i\circ e_{k}=\nu_ie_{k+1}, 1\leq i\leq s_1.$ Then by using (3) we obtain
  $\nu_i e_{k+2}=e_1\circ (f_i\circ e_{k})=(e_1\circ f_i)\circ e_{k}-e_1\circ (e_{k}\circ f_i)=0,$ which implies $\nu_i=0,$ i.e., $f_i\circ e_{k}=0$
  for all $1\leq i \leq s_1.$ Now by Lemma \ref{lemma} we have $f_i\circ e_{j}=0$ for $1\leq i \leq s_1 \ \hbox{ and } \ k\leq j\leq  n-p-1.$ Thus
  $\ll_{k+2}=\ll_1\circ \ll_{k+1}=\langle e_{k+2} \rangle, \ \ll_{k+3}=\langle e_{k+3}\rangle, \dots$ and $ \ll_{n-p}=\langle e_{n-p} \rangle.$
  However, if we had these equalities the vector $f_{p}$ would not be obtained.

Since $\ll_{i+1}=\ll_1\circ \ll_{i}=\langle e_{i+1},f_1\circ e_i,\dots,f_{s_1}\circ e_i\rangle$ and $i(f_j\circ e_i)=(f_j\circ e_{i-1})\circ e_1$ for
 $1\leq j \leq s_1,$ it follows that $\dim (\langle f_1\circ e_i,\dots,{f_{s_1}}\circ e_i\rangle) \leq \dim (\langle f_1\circ e_{i-1},\dots,
 {f_{s_1}}\circ e_{i-1}\rangle),$ that is, $\dim \ll_{i+1}\leq \dim \ll_i$ for $1\leq i \leq n-p-1.$ Hence it is proved that $s_{n-p}\leq \dots\leq s_2\leq s_1,$ which is the desired conclusion.
 \end{itemize}
\end{proof}

Note that we have actually proved that the natural gradation of our algebras is as follows:
$$\ll_1=\langle e_1,f_1,\dots,f_{s_1}\rangle, \  \ll_2= \langle e_2,f_{s_1+1},\dots,f_{s_1+s_2}\rangle, \dots, \ \ll_i=\langle e_i,f_{s_{i-1}+1},\dots,f_{s_{i-1}+s_i}\rangle$$ and we have written $f_i \circ f_j= \alpha_{ij}e_2$ for $1 \leq i,j \leq s_1.$

\section{Main Result}
Due to the above section we actually know the gradation of our algebra and we have obtained information on some structural constants. Next results complete
the information on the structural constants come to there are only one family in the classification of these kind
of Zinbiel algebras.

The following proposition presents additional multiplication law for Zinbiel algebra from Proposition \ref{s<s}.

\begin{prop}\label{fe1} Under the above assumptions the following multiplications hold for all $1\leq t \leq n-p$:
$$f_{s_1+\cdots+s_{t-1}+i}\circ e_1=\left\{ \begin{array}{cl}
f_{s_1+\dots+s_t+i} &\textrm{ for }1\leq i \leq s_{t+1}\\
0 & \textrm{ for } s_{t+1}+1\leq i \leq s_t\\
\end{array} \right.$$
\end{prop}

\begin{proof}

\

Since $\ll_2=\ll_1\circ \ll_1 =\langle e_2,f_1 \circ e_1,\dots, f_{s_2}\circ e_1,\dots, f_{s_1}\circ e_1 \rangle,$  there exist $1\leq n_1,\dots,n_{s_2}\leq s_1$
such that $\langle f_{n_1}\circ e_1,\dots,f_{n_{s_2}}\circ e_1 \rangle = \langle f_{s_1+1},\dots,f_{s_1+s_2}\rangle.$

On the one hand consider  the basis permutation $f^{(1)}_i=f_{n_i},\  \ f_{n_i}^{(1)}=f_i$ for all $1\leq i \leq s_2,$ where the remaining vectors stay unchanged. Then $
\langle f^{(1)}_1\circ e_1,\dots,f^{(1)}_{s_2}\circ e_1\rangle =\langle f_{s_1+1},\dots,f_{s_1+s_2}\rangle.$

 On the other hand we can make the following basis transformation:
$$f^{(1)}_{s_1+i}=f^{(1)}_i \circ e_1 \textrm{ for } 1\leq i \leq s_2.$$
Finally, if $\displaystyle f_{i}^{(1)}\circ e_1=\sum_{k=1}^{s_2} \alpha_k f_{s_1+k}^{(1)}$ for $s_2+1\leq i \leq s_1,$ let us make next transformation
 $$f^{(2)}_{i}=f^{(1)}_{i}-\displaystyle \sum_{k=1}^{s_2} \alpha_k f_{s_1+k}^{(1)} \hbox{ for } s_2+1\leq i \leq s_1.$$ Hence we obtain
$f^{(2)}_{i}\circ e_1=\displaystyle f^{(1)}_{i}\circ e_1-\sum_{k=1}^{s_2} \alpha_k f_{s_1+k}^{(1)}\circ e_1=0.$

Once these basis transformations are done, we can assert
$$f_{i}\circ e_1=\left\{ \begin{array}{cl}
f_{s_1+i} &\textrm{ for }1\leq i \leq s_{2}\\
0 & \textrm{ for } s_{2}+1\leq i \leq s_1\\
\end{array}\right.$$

Analogously, by means of appropriate basis transformation we come to
$$f_{s_1+\cdots+s_{t-1}+i}\circ e_1=\left\{ \begin{array}{cl}
f_{s_1+\dots+s_t+i} &\textrm{ for }1\leq i \leq s_{t+1}\\
0 & \textrm{ for } s_{t+1}+1\leq i \leq s_t\\
\end{array}\right.$$
for all $1\leq t \leq m$ and $1 \leq m\leq n-p-2.$

Moreover since
$$\begin{array}{l}
\langle e_{m+2},f_{s_1+\dots+s_{m+1}+1},\dots,
f_{s_1+\dots+s_{m+2}}\rangle=\ll_{m+2}\supseteq \ll_{m+1}\circ \ll_1=\\
=\langle e_{m+2},f_{s_1+\dots+s_m+1}\circ e_1,\dots, f_{s_1+\dots+s_m+s_{m+1}}\circ e_1\rangle
\end{array}$$ there exist $n_1,\dots,n_{s_{m+2}}$ with $1\leq n_1,\dots,n_{s_{m+2}}\leq s_{m+1}$
such that
$$\langle f_{s_1+\dots+s_m+n_1}\circ e_1,\dots,f_{s_1+\dots+s_m+{n_s}_{m+2}}\circ e_1 \rangle = \langle f_{s_1+\dots+s_{m+1}+1},\dots,
f_{s_1+\dots+s_{m+2}}\rangle.$$

It suffices to make some basis transformations more to prove the lemma. Once the following vectors are changed
$$f^{(1)}_j=f_{n_j},\, f^{(1)}_{s_1+j}=f_{s_1+n_j},\dots,f^{(1)}_{s_1+\dots+s_m+j}=f_{s_1+\dots+s_m+n_j},$$
$$f^{(1)}_{n_j}=f_j,\, f^{(1)}_{s_1+n_j}=f_{s_1+j},\dots,f^{(1)}_{s_1+\dots+s_m+n_j}=f_{s_1+\dots+s_m+j},$$
for $1\leq j \leq s_{i+1},$ we obtain
$$f^{(1)}_{s_1+\dots+s_{k-1}+i}\circ e_1=f^{(1)}_{s_1+\dots+s_k+i} \textrm{ for }
1\leq i \leq s_{k+1} \hbox{ and } 1\leq k \leq m.$$

Therefore $\langle f^{(1)}_{s_1+\dots+s_m+1}\circ e_1,\dots,f^{(1)}_{s_1+\dots+s_m+s_{m+2}}\circ e_1 \rangle = \langle f_{s_1+\dots+s_{m+1}+1},\dots,
f_{s_1+\dots+s_{m+2}}\rangle.$

After this other one $f_{s_1+\dots+s_{m+1}+i}^{(1)}=
f^{(1)}_{s_1+\dots+s_m+i}\circ e_1$ for $1\leq i \leq s_{m+2}$  we have
$$f^{(1)}_{s_1+\dots+s_{k-1}+i}\circ e_1=f^{(1)}_{s_1+\dots+s_k+i} \textrm{ for }
1\leq i \leq s_{k+1} \hbox{ and for all }  \ 1\leq k \leq m+1.$$

As $n_1,\dots,n_{s_m+2}<s_{m+1}$ the following zero multiplications stay unchanged:
$$f_{s_1+\dots+s_{k-1}+s_{k+1}+i}\circ e_1=0 \textrm{ for }
1\leq i \leq s_k-s_{k+1}\ \hbox{ and for all } \ 1\leq k \leq m.$$

Now if $f_{s_1+\dots+s_{m}+i}\circ e_1=\displaystyle \sum_{k=1}^{s_{m+2}} \alpha_k
f_{s_1+\dots+s_{m+1}+k}$ for $s_{m+2}+1\leq i \leq s_{m+1}$ we take
$$f_{s_1+\dots+s_{m}+i}^{(1)}=f_{s_1+\dots+s_{m}+i}-
\displaystyle\sum_{k=1}^{s_{m+2}} \alpha_k f_{s_1+\dots+s_{m}+k}$$
yields
$f_{s_1+\dots+s_{m}+i}^{(1)}\circ e_1=0 \ \hbox{ for } \ s_{m+2}+1\leq i \leq s_{m+1}.$
\end{proof}

The following result is the generalization of the previous proposition.

\begin{prop}\label{fej} Let $\ll$ be an $n$-dimensional $p$-filiform Zinbiel algebra of first type. Then
$$f_{s_1+\dots+s_k+i}\circ e_j=\left\{
\begin{array}{cl}
\frac{1}{j!}f_{s_1+\dots+s_{k+j}+i} &\hbox{ for }\  1\leq i \leq s_{k+j+1}\\
0 & \hbox{ for } \ s_{k+j+1}+1\leq i \leq s_{k+j}\\
\end{array}\right.$$ when $1\leq j \leq n-p-k.$
\end{prop}

\begin{proof}
We proceed by induction on $j.$ For $j=1$ the products are directly obtained by Proposition \ref{fe1}. Assume that
$$f_{s_1+\dots+s_k+i}\circ e_j=\left\{
\begin{array}{cl}
\frac{1}{j!}f_{s_1+\dots+s_{k+j}+i} &\hbox{ for } \ 1\leq i \leq s_{k+j+1}\\
0 & \hbox{ for } \  s_{k+j+1}+1\leq i \leq s_{k+j}\\
\end{array}\right.$$  when $1\leq j \leq n-p-k$ and for each $2 \leq k \leq n-p-2$. Note that for $j \geq 2$ and $n-p-1 \leq k \leq n-p$ the product $f_{s_1+\dots+s_k+i}\circ e_j=0$ by properties of the natural gradation. We will prove these products for $j+1.$

$$f_{s_1+ \dots+ s_k+i} \circ e_{j+1}=f_{s_1+ \dots+ s_k+i} \circ (e_1 \circ e_j)=(f_{s_1+ \dots+ s_k+i} \circ e_1) \circ e_j- f_{s_1+ \dots+ s_k+i}(e_j \circ e_1).$$ Due to Proposition \ref{prode} we can write:
$$f_{s_1+ \dots+ s_k+i} \circ e_{j+1}=f_{s_1+ \dots+ s_{k+1}+i} \circ e_j - j f_{s_1+ \dots+ s_k+i} \circ e_{j+1}$$
and by induction hypothesis we have:
 $$(1+j)f_{s_1+ \dots+ s_k+i} \circ e_{j+1}=f_{s_1+ \dots+ s_{k+1}+i} \circ e_j =\left\{\begin{array}{cl}
\frac{1}{j!}f_{s_1+\dots+s_{k+j+1}+i} & \hbox{ for } \ 1\leq i \leq s_{k+j+2}\\
0 & \hbox{ for } \ s_{k+j+2}+1\leq i \leq s_{k+j+1}\\
\end{array}\right.$$
Therefore we obtain
$$f_{s_1+ \dots+ s_k+i} \circ e_{j+1}\left\{\begin{array}{cl}
\frac{1}{(j+1)!}f_{s_1+\dots+s_{k+j+1}+i} & \hbox{ for } \ 1\leq i \leq s_{k+j+2}\\
0 & \hbox{ for } \ s_{k+j+2}+1\leq i \leq s_{k+j+1}\\
\end{array}\right.$$
and the lemma is proved.
\end{proof}

The following lemmas present the rest of multiplications among the elements of the basis.
\begin{lem} Let $\ll$ be an $n$-dimensional $p$-filiform Zinbiel algebra of first type. Then $f_i\circ f_j =0$ for $1\leq i \leq s_1$ and $s_{n-p-1}+1\leq j \leq s_1.$
\end{lem}

\begin{proof} From $(4)$ we have $f_i\circ f_j=\gamma_{ij}e_2$ for $1\leq i,j\leq s_1.$ If we consider $s_{n-p-1}+1\leq j\leq s_1,$ then there exists $1\leq q \leq n-p-2$ such that $s_{q+1}+1\leq j\leq s_q.$ Moreover, by Lemma \ref{fej} we have $f_j\circ e_q=0.$ Therefore, for any $1\leq i \leq s_1$ we have
$$\gamma_{ij} C^{2}_{i+j+1} e_{q+2}= \gamma_{ij}e_2\circ e_q= (f_i\circ f_j)\circ e_q=f_i\circ (f_j\circ e_q)+f_i\circ (e_q\circ f_j)=0,$$ which implies $\gamma_{ij}=0$ for $1\leq i \leq s_1,$ $s_{q+1}+1\leq j \leq s_q$ and the lemma follows.

\end{proof}


\begin{lem} \label{lema1} Let $\ll$ be an algebra under the above conditions. Then
$$\begin{array}{l}
1) \ e_i \circ f_j=0 \hbox{ for } 1 \leq i \leq n-p \hbox{ and } 1 \leq j \leq p,\\
2) \ f_i \circ f_{s_1+i}=0 \hbox{ for } 1 \leq i \leq s_2.
\end{array}$$
\end{lem}

\begin{proof}
Proposition \ref{s<s}, Proposition \ref{fe1} and the properties of the nilpotent algebras make obvious the proof.

%
\end{proof}

\begin{lem} \label{lema2}
Under the above conditions, the following products are true:
$$f_{s_1+ \dots + s_t+i} \circ f_{s_1 + \dots +s_k+j}=\gamma_{ij}(t+k+1)!e_{t+k+2}$$
for $ \ 0 \leq k+t \leq n-p-2, \ 1 \leq i \leq s_{t+2} \hbox{ and } \ 1 \leq j \leq s_{k+2}.$
\end{lem}

\begin{proof}
We first prove that $f_{i} \circ f_{s_1 + \dots s_k+j}=\gamma_{ij}(k+1)!e_{k+2}, \ 0 \leq k\leq n-p-2, \ 1 \leq i \leq s_{1}, \ 1 \leq j \leq s_{k+2},$ by using induction on $k.$ For $k=1$ we have:
$$f_i \circ f_{s_1+j}=f_i \circ (f_j \circ e_1)=(f_i \circ f_j)\circ e_1 - f_i \circ (e_1 \circ f_j)= \gamma_{ij}e_2 \circ e_1= 2\gamma_{ij}e_3$$ and
$f_{s_1+j}\circ f_i=(f_j \circ e_1)\circ f_i=f_j \circ (e_1 \circ f_i)+f_j \circ (f_i \circ e_1)=f_j \circ f_{s_1+i}=2\gamma_{ji}e_3=-2\gamma_{ij}e_3.$
Hence we conclude by Lema \ref{lema1} that $f_i \circ f_{s_1+j}=2 \gamma_{ij}e_3$ for $1 \leq j \leq s_2$ and $s_2+1 \leq i \leq s_1.$

Assume $f_{i} \circ f_{s_1 + \dots +s_k+j}=\gamma_{ij}(k+1)!e_{k+2}$ for $1 \leq k\leq n-p-2, \ 1 \leq i \leq s_{1},$ and $ 1 \leq j \leq s_{k+1};$ we will prove it for $k+1,$ that is:
$$\begin{array}{ll}
f_{i} \circ f_{s_1 + \dots +s_{k+1}+j}&=f_i \circ (f_{s_1+ \dots + s_k+j} \circ e_1)=(f_i \circ f_{s_1+ \dots + s_k+j})\circ e_1= \gamma_{ij}(k+1)!e_{k+2}\circ e_1=\\
&=\gamma_{ij}(k+1)!(k+2)e_{k+3}.
\end{array}$$

It is remains to prove that $f_{s_1+ \dots + s_t+i} \circ f_{s_1 + \dots +s_k+j}=\gamma_{ij}(t+k+1)!e_{t+k+2}$ for $1 \leq j \leq s_{k+2}$ and $s_{t+2}+1 \leq i \leq s_1,$ but this is trivial by considering the product $ f_{s_1 + \dots s_{k+1}+j} \circ f_i.$

Finally let us see $f_{s_1+i} \circ f_{s_1 + \dots +s_{k}+j}=\gamma_{ij}(k+2)!e_{k+3}.$
$$f_{s_1+i} \circ f_{s_1 + \dots +s_{k}+j}=(f_i \circ e_1) \circ f_{s_1 + \dots +s_{k}+j}=f_i \circ (f_{s_1 + \dots +s_{k}+j} \circ e_1)=f_i \circ f_{s_1 + \dots+ s_{k+1}+j}=\gamma_{ij}(k+2)!e_{k+3}.$$
Thus, the proof is completed.
\end{proof}

\begin{lem} \label{lema3} Under the condition stated above,
$f_i \circ f_j=-f_j \circ f_i \ \hbox{ for all } \ 1\leq i,j \leq p.$
\end{lem}

\begin{proof}
Due to the previous results we only need to prove that the products
 $f_{s_1+ \dots + s_t+i} \circ f_{s_1+ \dots + s_k+j}$ are antisymmetric for  $1 \leq k,t \leq n-p-1,$ $1 \leq i \leq s_{t+1}$ and $1 \leq j \leq s_{k+1}.$ In this direction it is worthwhile to check the following property:
$$f_{s_1+ \dots + s_{k}+i} \circ f_i=f_{s_1+ \dots + s_{k-1}+i}\circ f_{s_1+i} \hbox{ for } 1 \leq i \leq s_k+1.$$

Since $f_{s_1+ \dots + s_{k}+i} \circ f_i=(f_{s_1+ \dots + s_{k-1}+i} \circ e_1) \circ f_i=f_{s_1+ \dots + s_{k-1}+i} \circ (f_i \circ e_1)  =f_{s_1+ \dots + s_{k-1}+i}\circ f_{s_1+i},$ for $1 \leq i \leq s_{k+1},$ hence $f_{s_1+ \dots + s_t+i} \circ f_{s_1+ \dots + s_k+j}=f_{s_1+ \dots + s_m+i} \circ f_{s_1+ \dots + s_{t+k-m}+j}.$ According to Lemma \ref{lema2} we assert that this expression equals to $\gamma_{ij}(m+t+k-m+1)!e_{m+t+k-m+2}=\gamma_{ij}(t+k+1)!e_{t+k+2}.$ Therefore
$$f_{s_1+ \dots + s_t+i} \circ f_{s_1+ \dots + s_k+j}=\gamma_{ij}(t+k+1)!e_{t+k+2}=-\gamma_{ji}(t+k+1)!e_{t+k+2}$$
with $1 \leq i,j \leq s_1.$
Finally by using again Lemma \ref{lema2}, it follows the proof.
\end{proof}

\begin{thm}\label{classif}
Let $\ll$ be a naturally graded $p$-filiform $n$-dimensional Zinbiel algebra of first type such that $n-p\geq 4$. Then $\ll$ is isomorphic to one algebra of the family $M^{\alpha} (s_1,\dots, s_{n-p},r):$
$$\left\{
\begin{array}{ll}
e_i \circ e_j=C_{i+j-1}^{j}e_{i+j},& 2\leq i+j \leq n-p\\[2mm]
f_i \circ e_1=f_{s_1+i},& 1\leq i\leq s_2\\[2mm]
f_{{s_1+\cdots+s_k}+i} \circ e_1=f_{s_1+\cdots+s_{k+1}+i},& 1\leq i\leq s_{k+2}, \ 1\leq k \leq n-p-2\\[2mm]
f_i \circ f_{i+1}=\alpha e_2,& 1\leq i\leq r-1\\[2mm]
f_{s_1+\cdots+s_k+i} \circ e_j=\frac{1}{j!} f_{s_1+\cdots + s_{k+j}+i},& 1\leq i\leq s_{k+j+1}, \ 1 \leq k+j \leq n-p-1 \\[2mm]
f_i \circ f_{s_1+\cdots+s_k+i+1}=\alpha (k+1)! e_{k+2},& 1\leq i\leq s_1\\[2mm]
f_{s_1+\cdots+s_k+i} \circ f_{s_1+\cdots+s_t+i+1}=\alpha (k+t+1)! e_{k+t+2},& 2\leq k+t\leq r-1,\ \ 1\leq i\leq r-1
\end{array}
\right.$$
with $\alpha\in\{0,1\}.$
\end{thm}

\begin{proof}
Once the structural constants $\gamma_{ij},$ with $1 \leq i,j \leq s_{n-p-1}$  are determined, the classification will be obtained due to previous propositions and lemmas. Note that it is enough determinate $\gamma_{ij}$ with $1 \leq i < j \leq s_{n-p-1},$ since in Lemma \ref{lema3} we have proved that  $\gamma_{ij}=-\gamma{ji} \ \forall i,j.$ That it, the matrix of the unknown structural constants is:

$$\left( \begin{array}{cccccc}
0 & \gamma_{12} & \gamma_{13} & \dots & \gamma_{1s_{n-p-2}} & \gamma_{1s_{n-p-1}}\\
 -\gamma_{12} & 0 & \gamma_{23} & \dots & \gamma_{2s_{n-p-2}} & \gamma_{2s_{n-p-1}}\\
- \gamma_{13} & -\gamma_{23} & 0 & \dots & \gamma_{3s_{n-p-2}} & \gamma_{3s_{n-p-1}}\\
\vdots & \vdots & \vdots & \vdots & \vdots & \vdots\\
-\gamma_{1s_{n-p-2}} & -\gamma_{2s_{n-p-2}}  & -\gamma_{3s_{n-p-2}} &  \dots & 0 & \gamma_{s_{n-p-2}s_{n-p-1}}\\
-\gamma_{1s_{n-p-1}} & -\gamma_{2s_{n-p-1}}  & -\gamma_{3s_{n-p-1}} &  \dots  & -\gamma_{s_{n-p-2}s_{n-p-1}} & 0\\
\end{array} \right)$$

\

It is useful distinguish the following cases:

\

\fbox{Case 1:} If $\gamma_{ij}=0  \hbox{ for all } i,j$ then the algebra $M^0(s_1,\dots,s_{n-p},0)$ is obtained.

\

\fbox{Case 2:} If there exist $i_0$ and $j_0$ such that $\gamma_{i_0,j_0}\neq 0,$ we proceed in the following way:

There is no loss of generality in assuming $\gamma_{12} \neq 0$ (use the basis change $f_1'=f_{i_0}$ and $f_2'=f_{j_0}$.) Let us first prove that
 $\gamma_{1i}'=0$ for $3 \leq i \leq s_{n-p-1}.$ Take the basis change:
\begin{align*}
e_i'&=e_i \ \hbox{ for }\ 1 \leq i \leq n-p,\\
f_i'&=\gamma_{12}f_i-\gamma_{1i}f_2 \ \hbox{ for } \ 3 \leq i \leq s_{n-p-1},\\
f_i'&=f_i \ \hbox{ for } \ 1 \leq i \leq 2 \ \hbox{ and } \ s_{n-p-1}+1 \leq i \leq s_1.\\
\end{align*}

The only significant changes in the law of $\ll$ with respect to the new basis are the products $f_i' \circ f_1$ and $f_1 \circ f_i'.$ When $3 \leq i \leq s_{n-p-1}$ these products are zero since:
$$f_1' \circ f_i'=f_1 \circ (\gamma_{12}f_i- \gamma_{1i}f_2)=\gamma_{12}f_1 \circ f_i - \gamma_{1i}f_1 \circ f_2=\gamma_{12}\gamma_{1i}e_2-\gamma_{1i}\gamma_{12}e_2=0.$$
As we know that $\gamma_{1i}=-\gamma_{i1},$ we conclude that $f_i' \circ f_1'$ are also zero.

Therefore the matrix of the structural constants is simplified as:
 $$\left( \begin{array}{cccccc}
0 & \gamma_{12} & 0 & \dots & 0 & 0\\
 0 & 0 & \gamma_{23} & \dots & \gamma_{2\, s_{n-p-2}} & \gamma_{2\, s_{n-p-1}}\\
 0 & -\gamma_{23} & 0 & \dots & \gamma_{3\, s_{n-p-2}} & \gamma_{3\, s_{n-p-1}}\\
\vdots & \vdots & \vdots & \vdots & \vdots & \vdots\\
0 & -\gamma_{2\, s_{n-p-2}}  & -\gamma_{3\, s_{n-p-2}} &  \dots & 0 & \gamma_{s_{n-p-2}\, s_{n-p-1}}\\
0 & -\gamma_{2\, s_{n-p-1}}  & -\gamma_{3\, s_{n-p-1}} &  \dots  & -\gamma_{s_{n-p-2}\, s_{n-p-1}} & 0
\end{array} \right).$$

Reiterating the same reasoning $s_{n-p-2}$-times by using the basis change:
$$f_i'=\gamma_{k\,k+1}f_i-\gamma_{ki}f_{k+1} \hbox{ for } k+2 \leq i \leq s_{n-p-1},$$
 we come to $\gamma_{ij} \neq 0$ if and only if $j=i+1.$


It is clear that one can take $\gamma_{i\,i+1}=1$ by using the basis change $f_i'=\frac{1}{\gamma_{i\, i+1}}f_i$ for $1 \leq i \leq s_{n-p-2}.$

In this direction, we have obtain the family $M^\alpha(s_1,\dots,s_{n-p},r),$ where $r$ is the subindex such that
$$\begin{cases}
\gamma_{j\, j+1}\neq 0 \hbox{ if }\  1 \leq j \leq r,\\
\gamma_{j\, j+1}=0 \hbox{ if } \ r < j \leq s_{n-p-2.}
\end{cases}$$

Finally we prove that all the obtained algebras are pairwise non isomorphic.

Of course, two isomorphic algebras have equal the first $n-p$ parameters. Otherwise their natural gradation would be different. Hence we are going to focus our attention in the para\-meter $r$ to study the isomorphic property.

Let us consider $M^{s_1,\dots,s_{n-p},r_1}$ and $M^{s_1,\dots,s_{n-p},r_2}$ such that $r_1 \neq r_2$ and let us denote $M^1$ and $M^2$ respectively for convenience. Suppose, contrary to our claim, that these algebras are isomorphic. Consequently $dim(Cent(M^{s_1, \dots, s_{n-p},r_1}))=dim(Cent(M^{s_1, \dots, s_{n-p},r_2})),$ that is $dim(Cent(M^1))-dim(Cent(M^2))=0.$

 By the properties of the gradation we can assert that $$\{ e_{n-p}, f_{s_1+s_2+ \dots+s_{n-p-1}+1}, \dots, f_{s_1+s_2+ \dots+s_{n-p-1}+s_{n-p}}\}$$ are in the center of the two algebras. Moreover by the law of the two algebras we deduce that
$$\begin{array}{l}
\{f_{s_2+1}, \dots, f_{s_1},f_{s_1+s_3+1},\dots, f_{s_1+s_2},f_{s_1+s_2+s_4+1}, \dots,f_{s_1+s_2+s_3}, \dots,\\
f_{s_1+s_2+ \dots+s_{n-p-3}+s_{n-p-1}+1}, \dots, f_{s_1+ s_2+ \dots+ s_{n-p-1}+s_{n-p}}\}
\end{array}$$ are elements of the center of them. Then it suffices to study if
$$f_{s_1+ \dots+ s_{n-p-2}+s_{n-p}+1}, \dots, f_{s_1+ \dots+ s_{n-p-2}+s_{n-p-1}}$$ are in the center of $M^1$ and $M^2$.

 By the properties of the gradation and the law of the considered algebras we can assert that the products of the above vectors might be no equal to zero if and only if we consider the vector $f_i$ for $1 \leq i \leq s_1.$ In other words, we only need to calculate the following products:
$$ \ f_i \circ f_{s_1+ \dots + s_{n-p-2}+s_{n-p}+j}=\gamma_{r_{k}-1\,r_k}(n-p-1)!(\gamma_{r_k-1\,r_k}\gamma_{ij}-\gamma_{r_k-1\,j}\gamma_{ir_k}-\gamma_{r_k-1\,i}\gamma_{r_k\,j})e_{n-p},
\quad (6)$$ such that  $r_k=r_1$ if we are working with $M^1,$ and $r_k=r_2$ otherwise. In order to know these products it is convenient to distinguish the following cases:

If \fbox{\textbf{ $i=r_k-1$}} all these products are equal zero since:
 $$\begin{array}{l}
 f_{r_k-1} \circ f_{s_1+ \dots + s_{n-p-2}+s_{n-p}+j}=\\[1mm]
\qquad \qquad =\gamma_{r_{k}-1\,r_k}(n-p-1)!(\gamma_{r_k-1\,r_k}\gamma_{r_k-1\,j}-\gamma_{r_k-1\,j}\gamma_{r_k-1\,r_k}-\gamma_{r_k-1\,r_k-1}\gamma_{r_k\,j})e_{n-p}=\\[1mm] \qquad \qquad =\gamma_{r_{k}-1\,r_k}(n-p-1)!(\gamma_{r_k-1\,j}-\gamma_{r_k-1\,j})e_{n-p}=0
 \end{array}$$
 for each $j.$

 \

If \fbox{\textbf{ $i \neq r_k-1$}} we can take $j=r_k$ or $j \neq r_k.$

 On the one hand if $j=r_k$ we also have every products equal zero since:
$$(\gamma_{r_k-1,r_k}\gamma_{ir_k}-\gamma_{r_k-1,r_k}\gamma_{ir_k}-\gamma_{r_k-1,i}\gamma_{r_kr_k})=\gamma_{r_k-1,r_k}\gamma_{ir_k}-\gamma_{r_k-1,r_k}\gamma_{ir_k}$$
  and we know that $i \neq r_k-1$ and $i \neq r_k+1.$

  On the other hand, when $j \neq r_k,$ it is worthwhile to consider the following cases:
 \begin{itemize}
 \item[$\bullet$] If $i=r_k,$ we conclude from $(6)$  that the product is equal to zero.

 \

 \item[$\bullet$] If $i \neq r_k$ then
$\gamma_{r_{k}-1,r_k}(n-p-1)!(\gamma_{r_k-1,r_k}\gamma_{ij}-\gamma_{r_k-1,j}\gamma_{ir_k}-\gamma_{r_k-1,i}\gamma_{r_kj}).$ As $i \neq r_k,$ $j \neq r_k$ and $i < j$ yields $\gamma_{r_k-1,j}\gamma_{ir_k}=0$ and $\gamma_{r_k-1,i}\gamma_{r_kj}=0.$ Hence the products can be rewrite as follows:
 $$f_i \circ f_{s_1+ \dots + s_{n-p-2}+s_{n-p}+j}=\gamma_{r_{k}-1,r_k}^2\gamma_{ij}(n-p-1)!e_{n-p}.$$

 Moreover we have proved that $\gamma_{ij}=0$ if $i>r_k-1$ and $j \neq i+1$ (remain $r_1$ for $M^1$ and $r_2$ for $M^2$), hence:
 $$f_i \circ f_{s_1+ \dots + s_{n-p-2}+s_{n-p}+i+1}=\gamma_{r_{k}-1,r_k}^2\gamma_{ij}(n-p-1)!e_{n-p} \neq 0 \quad \hbox{ for }1 \leq i \leq r_k-2.$$
 \end{itemize}
 Note that we have actually proved that:

$$\begin{array}{r}
<f_{s_1+ \dots+s_{n-p-2}+s_{n-p}+1}, \dots ,f_{s_1+ \dots+s_{n-p-2}+r_1-2}> \nsubseteq Cent(M^1),\\[2mm]
<f_{s_1+ \dots+s_{n-p-2}+r_1-1},f_{s_1+ \dots+s_{n-p-2}+r_1}, \dots ,f_{s_1+ \dots+s_{n-p-2}+s_{n-p-1}}> \subseteq Cent(M^1),\\[2mm]
<f_{s_1+ \dots+s_{n-p-2}+s_{n-p}+1}, \dots, f_{s_1+ \dots+s_{n-p-2}+r_2-2}> \varsubsetneq Cent(M^2),\\[2mm]
<f_{s_1+ \dots+s_{n-p-2}+r_2-1},f_{s_1+ \dots+s_{n-p-2}+r_2}, \dots ,f_{s_1+ \dots+s_{n-p-2}+s_{n-p-1}}> \subseteq Cent(M^2).
\end{array}$$

  Therefore we conclude
 $$\begin{array}{ll}
 Cent(M^1)&=<e_{n-p}, f_{s_2+1}, \dots, f_{s_1},f_{s_1+s_3+1},\dots, f_{s_1+s_2},f_{s_1+s_2+s_4+1}, \dots,\\
 & f_{s_1+s_2+s_3}, \dots,f_{s_1+s_2+ \dots+s_{n-p-3}+s_{n-p-1}+1}, \dots, f_{s_1+ s_2+ \dots+ s_{n-p-3}+s_{n-p-2}},\\
  &f_{s_1+ \dots+s_{n-p-2}+r_1-1},f_{s_1+ \dots+s_{n-p-2}+r_1}, \dots f_{s_1+ \dots+s_{n-p-2}+s_{n-p-1}},f_{s_{1}+\dots+s_{n-p-1}+1}, \dots,\\
   &f_{s_1+ \dots+s_{n-p}}>.
 \end{array}$$
and
 $$\begin{array}{ll}
  Cent(M^2)&=<e_{n-p}, f_{s_2+1}, \dots, f_{s_1},f_{s_1+s_3+1},\dots, f_{s_1+s_2},f_{s_1+s_2+s_4+1}, \dots,\\
  &f_{s_1+s_2+s_3}, \dots,f_{s_1+s_2+ \dots+s_{n-p-3}+s_{n-p-1}+1}, \dots, f_{s_1+ s_2+ \dots+ s_{n-p-3}+s_{n-p-2}},\\
  &f_{s_1+ \dots+s_{n-p-2}+r_2-1},f_{s_1+ \dots+s_{n-p-2}+r_2}, \dots f_{s_1+ \dots+s_{n-p-2}+s_{n-p-1}},f_{s_{1}+\dots+s_{n-p-1}+1}, \dots,\\
   &f_{s_1+ \dots+s_{n-p}}>.
   \end{array}$$

 It follows that $$dim(Cent(M^1))-dim(Cent(M^2))=s_{n-p-1}-r_1+2-s_{n-p-1}+r_2-2=r_2-r_1 \neq 0,$$ which contradicts the initial assumption and proves the theorem.
\end{proof}

\end{document}